\documentclass[11pt]{article}
\usepackage{pdfsync}
\usepackage{amsfonts}
\usepackage{amsfonts}  
\usepackage{amsthm}
\usepackage{amsmath}
\usepackage{color}

\title{A  Liouville theorem for fully nonlinear  problems  with  infinite boundary conditions and applications} 

\author{Isabeau Birindelli \\
Dipartimento di Matematica, Sapienza Universit\`a\  di Roma
\and
  Fran\c{c}oise Demengel\\
  D\'epartement de Math\'ematiques,
Universit\'e\  de Cergy-Pontoise
  \and
   Fabiana  Leoni\\ 
   Dipartimento di Matematica, Sapienza Universit\`a\  di Roma}

\date{}

\catcode`@=11 \@addtoreset{equation}{section} \catcode`@=12

\newtheorem{theo}{Theorem}[section]
\newtheorem{prop}[theo]{Proposition}

\newtheorem{lemme}[theo]{Lemma}

\def\R{\mathbb  R}
\def\grad{\nabla}

\begin{document}

\maketitle
\begin{abstract}
We prove a Liouville type classification theorem in half-spaces for  infinite boundary value problems related to fully nonlinear, uniformly elliptic operators. We then apply the result in order to obtain gradient boundary blow up rates for ergodic functions in bounded domains related to degenerate/singular operators, and, as a further consequence, we deduce the uniqueness of the ergodic functions.
\smallskip

\end{abstract}

\section{Introduction}

A classical result from harmonic function theory, see e.g. \cite{ABR}, states that if $v$ is a positive harmonic function in the upper half-space $\R^N_+=\{ x=(x',x_N)\in \R^N\ :\ x_N>0\}$, vanishing on $\partial \R^N_+$, then, necessarily, $v(x)=c\, x_N$ for some $c>0$.
 
 By performing the change of unknown
 $$ v(x)= e^{-u(x)}\, ,$$
 the same result may be read in terms of $u$ as
 $$
 \left\{
 \begin{array}{c}
 -\Delta u +|\nabla u|^2=0 \ \hbox{ in } \R^N_+\\[1ex]
 u=+\infty \ \hbox{ on } \partial \R^N_+
 \end{array}
 \right.
 \ \Longleftrightarrow \ u(x)= -\log x_N + c\, ,\ c\in \R\, .
 $$
 We are interested here in extending the above result to more general equations. Namely, 
  we focus on the second order, fully nonlinear infinite boundary value problem
  \begin{equation}\label{eqergo}\left\{ \begin{array}{cl}
   -F( D^2 u) +| \nabla u |^\beta = 0& {\rm in} \ \R_+ ^N\\[2ex]
    u( x^\prime,  0) = +\infty &  {\rm on} \ \partial\R_+ ^N
    \end{array}
\right.
\end{equation} 
where  $\beta >1$ and $F:{\mathcal S}_N\to \R$ is a positively homogeneous  uniformly elliptic operator, that is a continuous function defined on the space of $N\times N$-square symmetric matrices ${\mathcal S}_N$, positively homogeneous of degree 1, and satisfying 
\begin{equation}\label{ue}
a\, {\rm tr} (P)\leq F(X+P)-F(X)\leq A\, {\rm tr} (P)\, ,
\end{equation}
for all $X, P\in {\mathcal S}_N$, with $P\geq O$, for some positive constants $A\geq a>0$. In \eqref{eqergo}, $D^2u$ stands for the hessian matrix of the unknown function $u$, and $\nabla u$ for its gradient.
\smallskip

We assume that $u\in C(\R^N_+)$ is a viscosity solution of \eqref{eqergo},  satisfying further
\begin{itemize}

\item[(H1)]  the boundary condition is uniformly satisfied, namely
  $$\forall \, M>0\, ,\ \exists\,  \epsilon_M>0\ \hbox{ such that } u(x^\prime, x_N) \geq M, \  \forall\,  x^\prime \in \R^{N-1}\ \hbox{ and } x_N\leq \epsilon_M\, ,$$
  
\item[(H2)]  $u$ is bounded for $x_N$ bounded away from zero and from infinity, that is
   $$ \forall \, \delta\in (0,1)\, ,\ \exists\, C_\delta>0\ \hbox{ such that } |u(x^\prime, x_N)| \leq C_\delta, \  \forall\,  x^\prime \in \R^{N-1}\ \hbox{ and }  x_N\in \left[ \delta, \frac{1}{\delta}\right] $$
\end{itemize}
   
Our main result is the following Liouville type classification theorem.
   
    \begin{theo}\label{liouville} Let $1<\beta\leq 2$ and assume that $F$ satisfies \eqref{ue}. If $u\in C(\R^N_+)$ is a viscosity solution of \eqref{eqergo} satisfying assumptions {\rm (H1)} and {\rm (H2)}, then, for some $c\in \R$,
         $$u(x^\prime, x_N) \equiv u(x_N)=
         \left\{
\begin{array}{ll}
\displaystyle  \frac{1}{2-\beta}\left(\frac{F(e_N\otimes e_N)^{\frac{1}{2-\beta}}}{(\beta-1)\, x_N}\right)^{\frac{2-\beta}{\beta-1}} + c & \hbox{ if } \beta<2\, ,\\[4ex]
\displaystyle -F( e_N \otimes e_N) \log x_N + c & \hbox{ if } \beta=2\, .
 \end{array}\right.      $$
        \end{theo}

We observe that the logarithmic change of variable useful in the  case of Laplace operator cannot be used for a fully nonlinear operator $F$, even when $\beta=2$, since the new variable $v$ will be merely a subsolution and not a solution  of the homogeneous equation.
Theorem \ref{liouville} will be instead established by showing first that any viscosity solution $u$ actually is a monotone decreasing function of the only variable $x_N$, and then integrating the resulting ODE. The one dimensional symmetry of $u$ will be obtained in turn  by applying the well known so called sliding method, introduced in \cite{BHM}. Indeed, after showing that the comparison principle holds true for bounded viscosity sub- and  supersolutions in horizontal strips, any solution $u$ will be proved to satisfy the inequality
\begin{equation}\label{comp}
u\leq u_t\, ,
\end{equation}
where $u_t(x)=u(x+t\nu )$ stands for the translated function with respect to any direction $\nu=(\nu',\nu_N)$ with $\nu_N \leq 0$ and $t>0$. We emphasize that the proof  for Theorem \ref{liouville} is even more simple than the original proof given in \cite{BHM}, where inequality \eqref{comp} is first established for $t$ large, and then for every $t>0$ by a contradiction argument. In our case, since the comparison principle for solutions in horizontal strips hold true indipendently of the size of the solutions to be compared, inequality \eqref{comp} is established simultaneously for all $t>0$.

Theorem \ref{liouville} belongs to the large family of symmetry results for solutions of nonlinear elliptic equations. Besides \cite{BHM}, we just mention the works \cite{FPS, FMRS, BD}, and refer to the references therein, as recent symmetry results for solutions of semilinear, quasilinear and fully nonlinear equations respectively. In all previous results, the symmetry  is obtained for entire solutions or solutions vanishing on the boundary. Up to our knowledge, Theorem \ref{liouville} is the first application of the sliding method to infinite boundary value problems.

An interesting feature of the Liouville property for problem \eqref{eqergo} relies in its connection with the so called ergodic problem associated with fully nonlinear degenerate/singular operators. Let us recall that, given a $\mathcal{C}^2$ bounded domain $\Omega\in R^N$  and a continuous function
 $f\in \mathcal{C}(\Omega)$,   an ergodic constant related to $\Omega$ and $f$ is a constant $c_\Omega\in \R$  such that there exist  solutions $u\in \mathcal{C}(\Omega)$ (called ergodic functions) of the infinite boundary value problem 
\begin{equation}\label{ergodicO}
\left\{
\begin{array}{cl}
-|\nabla u|^\alpha F(D^2u)+|\nabla u|^\beta = f+ c_\Omega & \hbox{ in } \Omega\, ,\\[2ex]
u=+\infty & \hbox{ on } \partial \Omega\, ,
\end{array}
\right.
\end{equation}
where $F$ is a positively homogeneous   operator satisfying \eqref{ue} as before, and the exponents $\alpha$ and $\beta$ satisfy respectively $\alpha>-1$ and $\alpha+1< \beta \leq \alpha+2$.

Problem \eqref{ergodicO} has been originally studied  in the semilinear case $\alpha=0$ and $F(D^2u)=\Delta u$ in \cite{LL}, where the terminology "ergodic" has been introduced and its connection with a state constraint optimal stochastic control problem has been showed. Further contributions have been given in \cite{P}, and in \cite{LP}  for $p$-Laplace operator as principal part. 

In the fully nonlinear degenerate/singular setting, 
problem \eqref{ergodicO} has been recently studied in \cite{BDL1}, where it has been proved in particular that if $f$ is bounded and Lipschitz continuous, then  ergodic pairs $(c_\Omega, u)$ do exist. 
By assuming further that 
\begin{equation}\label{smoothFd}
F(\nabla d(x) \otimes \nabla d(x)) \quad \hbox{ is of class $\mathcal{C}^2$ for $x$ in a neighborhood of $\partial \Omega$,}
\end{equation}
where $d(x)$ denotes the distance function from $\partial \Omega$,  in \cite{BDL1} the ergodic constant $c_\Omega$ is proved to be unique and any ergodic function $u$ is shown to satisfy the asymptotic identities
\begin{equation}\label{asym1}
\lim_{d(x)\to 0} \frac{u(x)\, d(x)^\chi}{C(x)}=1  \mbox{ if}\  \chi>0, 
\end{equation}
and 
\begin{equation}\label{asym2}
 \lim_{d(x)\to 0} \frac{u(x)}{|\log d(x)|\, C(x)}=1   \hbox{ if}\  \chi=0\, ,
\end{equation}
where
$$\chi = \frac{ 2+ \alpha-\beta }{ \beta-1-\alpha }\, ,$$
and $C(x)$ is any $\mathcal{C}^2(\Omega)$ function satisfying, for $x$ in a neighborhood of $\partial \Omega$,    
\begin{equation}\label{C(x)}
\begin{array}{ccc}C(x)=\left((\chi+1)F(\grad d(x)\otimes\grad d(x))\right)^{\frac{1}{ \beta-\alpha-1}}\chi^{-1}\  &\mbox{if }& \chi>0,\\[1ex]
 C(x)=F(\grad d(x)\otimes\grad d(x))\quad & \mbox{if }& \chi=0.
 \end{array}
\end{equation}
Moreover,  in \cite{BDL2} the ergodic functions are proved to be $C^{1,\gamma}_{\rm loc}(\Omega)$ and to satisfy, in the case $\chi>0$ (i.e. $\beta<\alpha+2$),  the gradient boundary asymptotics 
\begin{equation}\label{gradbbu}
\lim_{d(x)\to 0}   \frac{d(x)^{\chi+1} \nabla u(x)\cdot \nabla d(x)}{C(x)}=
 -\chi \, .
\end{equation}
As a consequence, the uniqueness of the ergodic function is proved in this case, see Theorem 1.2 of \cite{BDL2}.

By keeping the same terminology for unbounded domains, Theorem \ref{liouville} states that the ergodic constant for $\Omega=\R^N_+$ and $f\equiv 0$ is $c=0$, and the ergodic function is unique and coincides with the one variable function $u(x_N)$ given in the statement of Theorem \ref{liouville}. Indeed, as proved in \cite{BDNodea, IS}, one has
$$
-|\nabla u|^\alpha F(D^2u)+ |\nabla u|^\beta =0 \ 
\Longleftrightarrow \ 
-F(D^2u)+|\nabla u|^{\beta-\alpha}=0
$$
Note that, in the present assumptions, one has $1<\beta-\alpha\leq 2$, and thus Theorem \ref{liouville} applies.
More than that, by  standard scaling arguments, see e.g. \cite{KSZ, LP}, Theorem \ref{liouville} may be used in the blow up analysis near the boundary for any ergodic function,  yielding gradient boundary blow up rates for ergodic functions in any bounded domain.

Therefore, we deduce from Theorem \ref{liouville}   asymptotic identities analogous to \eqref{gradbbu} also for the case $\beta=\alpha+2$, i.e. $\chi=0$. These imply in particular that the gradient of any ergodic function does not vanish near the boundary and, as a consequence, the strong maximum principle applies, see \cite{BDNodea}, and yields the uniqueness of the ergodic function.

The paper is organised as follows: in Section 2 we give the proofs of the comparison principle in strips and of Theorem \ref{liouville}; in Section 3 we prove identity \eqref{gradbbu} for the case $\chi=0$ and deduce the uniqueness of the ergodic function, referring to \cite{BDL3} for the proof.

\section{Proof of Theorem \ref{liouville}}        
 Theorem \ref{liouville} will be obtained as an easy consequence of the following comparison result in strips.
 
              \begin{prop}\label{strong}
                                      Suppose that $u$ and $v$ are bounded and Lipschitz continuous functions on the strip $\Sigma =\R^{N-1} \times (b, c)$, satisfying in the viscosity sense
                                      $$ -F( D^2 u) + | \nabla u |^\beta \leq  0\leq  -F( D^2 v) + | \nabla v |^\beta\qquad \hbox{ in } \Sigma\, . $$
                                      Then 
                                       $$ u(x) \leq v(x)+ \sup_{\partial \Sigma}   (u-v)\, , \qquad \forall\, x\in \Sigma\, .$$
                                       \end{prop}
\begin{proof}                                                                         
                                        Suppose by contradiction that 
                                        $$ S\, : =\sup_{\Sigma} (u-v) > \sup_{\partial \Sigma} ( u-v)\, ,$$
 and let us set $ w(x)= v(x)+ S$.  Then, one has $u(x)\leq w(x)$ for all $x\in \Sigma$ and there exists a sequence  $\left\{x_j=\left(x'_j, (x_N)_j\right)\right\}\subset \Sigma$ such that 
                                    $ (u-w)(x_j) \rightarrow 0$. Up to a (not relabeled) subsequence, we have  $(x_N)_j \rightarrow \bar x_N \in [ b, c]$. 
                                    
We define $u_j( x^\prime, x_N) = u( x^\prime + x^\prime_j, x_N)$ and 
                                     $ w_j( x^\prime, x_N) = w( x^\prime + x^\prime_j, x_N)$ . 
                                      By Ascoli-Arzel\`a\ Theorem, again up to a subsequence,  $u_j$ and $w_j$ converge locally uniformly  to some $\bar u$ and $\bar w$, which  are  sub and supersolution respectively. 
                                       Furthermore, 
                                        $$\bar u ( 0, \bar x_N) = \lim_{j\to \infty} u \left( x^\prime _j, (x_N)_j\right) = \lim_{j\to \infty} w\left( x^\prime _j, (x_N)_j\right) = \bar w( 0, \bar x_N)\, .$$
                                         Since                                         $\bar u \leq \bar w$ in $\Sigma$, $\bar u\leq \bar w +\sup_{\partial \Sigma} (u-v)-S<\bar w$ on $\partial \Sigma$  and $\bar u (0, \bar x_N)=\bar w (0, \bar x_N)$, we deduce that $\bar x_N\in (b,c)$ and, by the strong comparison principle, that
                                           $\bar u \equiv \bar w$  in $\overline{\Sigma}$, which yields a contradiction for  either $x_N =b$ or $x_N = c$.  
\end{proof}

\bigskip

\noindent \emph{ Proof of Theorem \ref{liouville}.}
         
The conclusion will be obtained by showing that $u$ is a monotone decreasing function depending on the only variable $x_N$.

Let us first observe that any solution $u$ of the equation in \eqref{eqergo} is locally $C^{1,\gamma}$ in $\R^N_+$ for some $\gamma\in (0,1)$ depending on $N, a, A$  and $\beta$, see e.g. \cite{BDL3}. Moreover, as proved e.g. in \cite{CDLP} and \cite{BDL1}, $u$ satisfies the local Lipschitz estimate
\begin{equation}\label{Lip}
|Du(x)|\leq \frac{C}{x_N^{\frac{1}{\beta-1}}}\, ,
\end{equation}
for a positive constant $C$ depending only on $N, a, A$ and $\beta>1$.

 Next, for a fixed unitary vector $\nu=(\nu',\nu_N)\in \R^N$ with $\nu_N<0$, and for $t>0$,  let us consider  the function   
 $$
 u_t(x)\, := u(x+t\, \nu)\, ,
 $$
 defined in  $\R^N_t\, :=\{ (x',x_N)\in \R^N\ :\ x_N>-t\, \nu_N\}$. Clearly, $u_t$ is a solution of the infinite boundary value problem \eqref{eqergo} in $\R^N_t$.
 
 We claim that $u\leq u_t$ in $\R^N_t$. 
 
 Indeed, by assumptions (H1) and (H2),  we have that, uniformly with respect to $x'\in \R^{N-1}$,
 $$
 u(x)-u_t(x)\to -\infty \ \hbox{ as } x_N\to -t\, \nu_N\, .
 $$
 Moreover, by estimate \eqref{Lip}, we also have
 $$
 u(x)-u_t(x)\to 0 \ \hbox{ as } x_N\to +\infty\, ,
 $$
  uniformly with respect to $x'\in \R^{N-1}$ as well. Therefore, for any $\epsilon>0$ there exists $\delta=\delta_{\epsilon, t}>0$ sufficiently small such that
 $$
 u(x)-u_t(x)\leq \epsilon \ \hbox{ for either } x_N\geq -t\, \nu_N+\frac{1}{\delta} \ \hbox{ or } x_N\leq -t\, \nu_N+\delta\, .
 $$
On the other hand, in the strip $\Sigma=\R^{N-1}\times \left( -t\, \nu_N+\delta, -t\, \nu_N+\frac{1}{\delta}\right)$,  $u$ and $u_t$ satisfy the assumptions of Proposition \ref{strong}, so that we obtain
$$
u(x)-u_t(x)\leq \epsilon \qquad \forall\, x\in \R^N_t\, .
$$
By letting $\epsilon\to 0$, we deduce the claim. 
 
 It then follows that $\frac{\partial u}{\partial \nu}\geq 0$ in $\R^N_+$ for all vectors $\nu=(\nu', \nu_N)$ with $\nu_N<0$.  In particular $\frac{\partial u}{\partial x_N}\leq 0$. Furthermore, 
 by the continuity of $Du$,  $\frac{\partial u}{\partial \nu}\geq 0$ is true also for vectors $\nu$ such that $\nu_N=0$. Hence, we obtain that $\frac{\partial u}{\partial x_i}=0$ for all $i=1,\ldots, N-1$  
 Therefore, $u(x)\equiv u(x_N)$ is a  nonincreasing  one variable  function satisfying in the viscosity sense
 $$
 \left\{
 \begin{array}{cl}
 F(u'' e_N\otimes e_N)=(-u')^\beta & \hbox{ for } x_N>0\\[2ex]
 u(0)=+\infty & 
 \end{array}
 \right.
 $$
This implies  that necessarily $u''\geq 0$, so that $u$ is a solution of the ODE
$$
u''=\frac{(-u')^\beta}{ F(e_N\otimes e_N)}\, .
$$
A direct integration of this equation and the imposition of the initial condition yield the result.
\hfill$\Box$

 \bigskip

                                        \section{Applications :  gradient boundary asymptotic behavior  and uniqueness of the ergodic functions}  
           
In this section we consider the ergodic problems in bounded domains associated with degenerate or singular fully nonlinear operators. Namely, we consider the infinite boundary value problem \eqref{ergodicO} for a smooth bounded domain $\Omega\subset \R^N$, and we refer to a solution $(c_\Omega ,u)$ as an ergodic pair. As recalled in the introduction, by assuming that $f\in C(\Omega)$ is bounded and Lipschitz continuous, $F$ is positively homogeneous and satisfies \eqref{ue}, $\alpha>-1$ and $\alpha+1<\beta\leq \alpha+2$, it is proved in \cite{BDL1} the existence of ergodic pairs $(c_\Omega ,u)$. Moreover, by further assuming the smoothness condition \eqref{smoothFd}, the ergodic constant $c_\Omega$ is unique and any ergodic function $u$ satisfies the asymptotic identities \eqref{asym1} and \eqref{asym2}, see \cite{BDL1}. Furthermore, in the case $\beta<\alpha+2$ any ergodic function is proved to satisfy also \eqref{gradbbu} and, consequently, to be unique, see \cite{BDL2}.

Here we are concerned  with the analogous results for the case $\beta=\alpha+2$, i.e. $\chi=0$.
  
\begin{theo}\label{asymgrad} Assume that $F$ satisfies \eqref{ue} and \eqref{smoothFd},  let $f\in {\mathcal C}(\Omega)$ be  bounded and let  $u\in {\mathcal C}(\Omega)$ be a solution of problem \eqref{ergodicO} with $\beta=\alpha+2$. Then, one has
 \begin{equation}\label{asymgrad1}
 \lim_{d(x)\to 0}   \frac{d(x) \nabla u(x)\cdot \nabla d(x)}{C(x)}=
 -1 \, .
 \end{equation}
\end{theo}
\begin{proof}
We recall that, by Theorem 6.3 of \cite{BDL1}, in the present assumptions any ergodic function $u$ satisfies the global inequalities
\begin{equation}\label{esti1}
-C(x)\,  \log d(x) -D\leq u(x)\leq -C(x)\, \log d(x) +D \qquad \forall\, x\in \Omega\, ,
\end{equation}
for a positive constant $D>0$ depending on the data and on $u$ itself.

Let us fix $x_0\in \partial \Omega$ and, for $\delta>0$, let us consider the rescaled function
$$
u_\delta (\zeta)=u(x_0+\delta\, \zeta)+C(x_0)\, \log \delta\, ,
$$
defined for $\zeta$ belonging to the translated and rescaled domain $O_\delta= \frac{\Omega -x_0}{\delta}$.

We observe that, as $\delta\to 0$, the domains $O_\delta$ approach the limiting halfspace $H\, :=\{ \zeta\in \R^N\ :\ \zeta\cdot \nabla d(x_0)>0\}$ and, by estimates \eqref{esti1} and the Lipschitz regularity in $\overline \Omega$ of $C(x)$, the functions $u_\delta$ are uniformly bounded for $\delta$ small and $\zeta$ in any compact subset of $H$.

Moreover, a direct computation shows that the functions $u_\delta$ are solutions of the equation
$$
-|\nabla u_\delta|^\alpha F(D^2 u_\delta) + |\nabla u_\delta|^{\alpha+2} = \delta^{\alpha+2} \left[ f(x_0+\delta\, \zeta)+c_\Omega\right]\quad \hbox{ in } O_\delta\, .
$$
By Theorem 1.1 of \cite{BDL2}, it follows that the sequence $\{u_\delta\}$ is uniformly bounded in $C^{1,\gamma}_{\rm loc} (H)$, for some $\gamma\in (0,1)$ depending on the data. Hence, by a standard diagonal procedure, there exists a  $v\in \mathcal{C}^{1,\gamma}_{\rm loc}(H)$ satisfying in the viscosity sense
$$
-|\nabla v|^\alpha F(D^2v)+|\nabla v|^{\alpha+2}=0\qquad \hbox{ in } H\, ,
$$
and such that
$$
u_\delta (\zeta)\to v(\zeta)\quad \hbox{  in $\mathcal{C}^{1}_{\rm loc}(H)$ as $\delta\to 0$.}
$$
By the results of \cite{IS},  it follows that actually $v$ is a viscosity solution (in the standard sense) of 
$$
-F(D^2v)+|\nabla v|^2=0\qquad \hbox{ in } H\, .
$$
Moreover, estimates \eqref{esti1} imply that $v$ satisfies the inequalities
$$
-D\leq v(\zeta) +C(x_0)\, \log (\nabla d(x_0)\cdot \zeta)\leq D\qquad \forall\, \zeta\in H\, .
$$
It then follows that  $v$ is a solution of the infinite boundary value problem
$$
\left\{ \begin{array}{cl}
   -F( D^2 v) +| \nabla v |^2 = 0& {\rm in} \ H\\[1ex]
    v= +\infty &  {\rm on} \ \partial H    \end{array}
\right.
$$
satisfying further assumptions (H1) and (H2). By Theorem \ref{liouville}, after a rotation,  we then deduce
$$
v(\zeta)= -F(\nabla d(x_0)\otimes \nabla d(x_0))\, \log (\nabla d(x_0)\cdot \zeta)+c= -C(x_0)\, \log (\nabla d(x_0)\cdot \zeta)+c\, ,
$$
for some $c\in \R$. Since $u_\delta\to v$ in $\mathcal{C}^1_{\rm loc}(H)$, we further obtain that, locally uniformly in $H$, one has, as $\delta\to 0$,
$$
\nabla u_\delta (\zeta)=\delta\, \nabla u(x_0+\delta\, \zeta)\to \nabla v(\zeta)=-C(x_0)\frac{\nabla d(x_0)}{\nabla d(x_0)\cdot \zeta}\, .
$$
This yields that
$$
\lim_{\delta\to 0} \frac{ d(x_0+\delta\, \zeta) \nabla u(x_0+\delta\, \zeta)\cdot \nabla d(x_0+\delta\, \zeta)}{C(x_0+\delta\, \zeta)}=-|\nabla d(x_0)|^2=-1\, ,
$$
and since all  convergences are uniform with respect to $x_0\in \partial \Omega$, we finally obtain the conclusion.
\end{proof}

Theorem \ref{asymgrad} implies, in particular, that $\nabla u\neq 0$ in a neighborhood of $\partial \Omega$, for any ergodic function $u$. This in turn enables the use of the strong maximum principle proved in \cite {BDNodea}, which yields the uniqueness of the ergodic function. Applying exactly the same proof of Theorem 1.2 of \cite{BDL2}, we deduce the following uniqueness result.

\begin{theo}\label{unique} Let $\Omega\subset \R^N$ be a bounded domain of class $C^2$ and let $F$  satisfy  \eqref{ue} and \eqref{smoothFd}. Assume further that $\alpha >-1$, $\alpha+1<\beta \leq \alpha +2$ and that $f\in C(\Omega)$ is  bounded. Then, up to additive constants,  problem (\ref{ergodicO}) has at most one  solution, provided that,  when   $\alpha \neq 0$, 
 $\sup_\Omega f <- c_\Omega$ and $\partial \Omega$ is connected.  
\end{theo}

                        \end{document}